\documentclass[12pt,a4paper]{amsart}

\usepackage[T1]{fontenc}
\usepackage{amssymb,amsmath,latexsym,amsthm,paralist,a4,mathrsfs,dsfont}
\usepackage{tikz-cd}
\usepackage{xspace}
\usepackage{txfonts}
\usepackage{enumitem}
\usetikzlibrary{decorations.markings,decorations.pathmorphing}
\usepackage{microtype}
\makeatletter
\tikzcdset{
  open/.code     = {\tikzcdset{hook, circled};},
  closed/.code   = {\tikzcdset{hook, slashed};},
  open'/.code    = {\tikzcdset{hook', circled};},
  closed'/.code  = {\tikzcdset{hook', slashed};},
 circled/.code  = {\tikzcdset{markwith = {\draw (0,0) circle (.375ex);}};},
 slashed/.code  = {\tikzcdset{markwith = {\draw[-] (-.4ex,-.4ex) -- (.4ex,.4ex);}};},
  markwith/.code ={
    \pgfutil@ifundefined
    {tikz@library@decorations.markings@loaded}
    {\pgfutil@packageerror{tikz-cd}{You need to say 
      \string\usetikzlibrary{decorations.markings} to use arrows with markings}{}}{}
    \pgfkeysalso{/tikz/postaction = {
      /tikz/decorate,
      /tikz/decoration={markings, mark = at position 0.5 with {#1}}}
    }
  },
}
\makeatother

\usepackage[headings]{fullpage}
\usetikzlibrary{babel}
\usetikzlibrary{positioning,calc}
\usepackage{varioref}
\usepackage[pdfpagelabels, pdftex]{hyperref}
\hypersetup{
  pdftitle={valuative criteria},
  pdfauthor={Magdalena Bauer},
  pdfsubject={},
  pdfkeywords={},
  colorlinks=true,    
  linkcolor=red,     
  citecolor=blue,     
  filecolor=blue,      
  urlcolor=blue,       
  breaklinks=true,
  bookmarksopen=true,
  bookmarksnumbered=true,
  pdfpagemode=UseOutlines,
  plainpages=false,
  unicode=true
  }
\usepackage[capitalise]{cleveref}
\usepackage{changes}
\makeatletter
\newcommand{\zerounderset}[3][\mathord]{%
  #1{\vtop{
    \let\\\cr
    \baselineskip\z@skip\lineskip.25ex
    \ialign{\hidewidth$##$\hidewidth\crcr
      \omit$#3$\cr
      #2\crcr
    }%
  }}%
}
\makeatother

\newtheoremstyle{thm123}
  {}
  {}
  {\sl }
  {}
  {\bf}
  {.}
  {.5em}
  {}
\theoremstyle{thm123}

\newtheorem{theorem}{Theorem}[section]
\newtheorem*{theorem*}{Theorem}

\newtheorem{proposition}[theorem]{Proposition}
\newtheorem{lemma}[theorem]{Lemma}
\newtheorem*{lemma*}{Lemma}

\newtheoremstyle{def123}
  {}
  {}
  {\rm }
  {}
  {\bf}
  {.}
  {.5em}
  {}
\theoremstyle{def123}
\newtheorem*{example*}{Example}

\newtheorem{remark}[theorem]{Remark}

\newtheorem{definition}[theorem]{Definition}

\DeclareMathOperator{\Spv}{\mathrm{Spv}}
\DeclareMathOperator{\Spa}{\mathrm{Spa}}

\DeclareMathOperator{\fr}{\mathrm{Frac}}
\DeclareMathOperator{\m}{\mathfrak{m}}
\newcommand{\ups}{\mathrm{\upsilon}}

\newcommand{\ba}{\begin{align*}}
\newcommand{\cO}{{\mathscr O}}
\newcommand{\ea}{\end{align*}}
\newcommand{\com}{\widehat}
\newcommand{\ti}{\tilde}
\newcommand{\p}{^\prime}
\newcommand{\pp}{^{\prime\prime}}

\begin{document}

\hfuzz=4pt
\title{Valuative Criteria for adic spaces}
\author{Magdalena Bauer}
\email{mbauer@math.uni-frankfurt.de}
\date{\today}
\thanks{The author acknowledges support by Deutsche Forschungsgemeinschaft (DFG) through the Collaborative Research Centre TRR 326 "Geometry and Arithmetic of Uniformized Structures", project number 444845124.}
\address{Institut für Mathematik, Goethe--Universität Frankfurt, Robert-Mayer-Str. 6--8, 60325 Frankfurt am Main, Germany}

\begin{abstract}
We establish valuative criteria for separatedness and (partial) properness for morphisms of adic spaces, which account for both vertical and horizontal specializations. 
The lifting test diagrams used in Huber’s analytic valuative criteria are naturally adapted to vertical specializations. 
In the non-analytic setting, one must additionally include lifting conditions for horizontal specializations.
Combining these horizontal criteria with Huber’s vertical criteria yields unified valuative characterizations of separatedness, universal specialization, and (partial) properness for arbitrary adic spaces.
\end{abstract}

\maketitle

\tableofcontents

\section{Introduction}
Valuative criteria serve as an indispensable tool for reducing global geometric properties of morphisms, most notably separatedness and properness, to concrete lifting problems with respect to valuation rings. 
The approach goes back to Grothendieck’s classical valuative criterion for schemes, which characterizes properness in terms of the existence and uniqueness of lifts in diagrams involving spectra of valuation rings.

In the context of adic geometry, Huber established an analogue for analytic adic spaces. 
Given a morphism $X\rightarrow Y$ of such spaces, 
an affinoid field $(k,k^+)$ and a valuation ring $A\subseteq k^+$, his criterion tests properness through the existence and uniqueness of a lifting morphism in diagrams of the form \begin{center}
    \begin{tikzcd}
            {\Spa(k,k^+)} \arrow[d] \arrow[r]     & X \arrow[d] \\
        {\Spa(k, A)} \arrow[r]\arrow[ru, dotted] & Y . 
    \end{tikzcd}
\end{center} 
The image of the lower morphism $\Spa{(k,A)}\rightarrow Y$ consists precisely of vertical specializations of the points in the image of the upper morphism $\Spa{(k,k^+)}\rightarrow X$. 
Consequently, these test diagrams govern only vertical specializations.  
In the analytic setting, where all specializations are vertical (i.e., they preserve the support of the valuation), this lifting property fully captures the relevant topological behavior of the morphism.  

However, for general (non-analytic) adic spaces, the situation is more nuanced due to the occurrence of horizontal specializations, which change the support of the valuation. 
While Huber's established criteria remain necessary in this broader setting, they no longer suffice, as they fail to detect horizontal specializations. 
A complete characterization of morphisms for arbitrary adic spaces therefore requires a complementary approach that tracks horizontal specializations.

In this paper, we develop the missing horizontal valuative framework and thereby establish valuative criteria reflecting the full specialization structure of adic spaces.
To capture horizontal specializations, we introduce horizontal valuated valuation rings together with horizontal centers.
They allow us to state a complementary lifting problem: 
for a morphism $X\rightarrow Y$, one considers lifts in diagrams of the form \begin{center}
    \begin{tikzcd}
            {\Spa(k,k^+)} \arrow[d] \arrow[r]     & X \arrow[d] \\
        {\Spa(B, k^+)} \arrow[r]\arrow[ru, dotted] & Y ,
    \end{tikzcd}\end{center} 
where $B$ is a valuation subring of $k$ containing $k^+$. 
The image of the lower morphism now parametrizes horizontal specializations, in direct analogy to the vertical case.
    
The construction is guided by a structural result due to Huber \cite{Hu1993}: every specialization in an adic space factors as a vertical specialization followed by a horizontal specialization. 
This decomposition is the organizing principle of the paper. 
We first state the vertical criteria, then prove their horizontal analogues, and finally combine the two to obtain unified valuative criteria for arbitrary adic spaces.
These unified results consequently characterize separatedness, the universal specialization property, and (partial) properness in full generality.
 
The paper is organized as follows. Section \ref{sectionVvrings} recalls valuated valuation rings following \cite{Hu1993} and introduces their horizontal variants, which function as local testing objects for our criteria.
Section \ref{sectionverticalVc} is devoted to a slight generalization of the classical valuative criteria \cite[Section 1.3] {Hu1996} to all adic spaces by restricting exclusively to vertical properties.
Section \ref{sectionhorizontalVc} contains the new component, namely horizontal valuative criteria for adic spaces. 
Finally, Section \ref{sectionunifiedVc} combines both perspectives to gain unified criteria that account for any specialization in adic spaces. 

\section{Valuated valuation rings}\label{sectionVvrings}
We begin by recalling Huber’s notion of valuated valuation rings.
In the analytic setting all specializations are vertical, and Huber's original test objects are thus what we will call \emph{vertical} valuated valuation rings. 
To accommodate horizontal specializations, we introduce their respective counterparts.

\begin{definition}(Huber {\cite[Lemma 1.5.2 and p. 238]{Hu1993}})\begin{enumerate}
    \item A \emph{valuated local ring} is a pair $(A,A^+)$ consisting of a local ring $A$ equipped with a valuation $\ups$ whose support is the maximal ideal of $A$, and where $A^+$ is the subring of elements of value at most 1. 
    \item A \emph{local ring homomorphism of valuated local rings} $(A,A^+)$ and $(B,B^+)$ is a local ring homomorphism $f:A\rightarrow B$ such that $f(A^+)\subseteq B^+$ and the induced ring homomorphism $A^+\rightarrow B^+$ is local. 
    \item A \emph{valuated valuation ring} is a valuated local ring $(A,A^+)$ where $A$ is a valuation ring. 
    \end{enumerate}
\end{definition}
To relate these rings to adic spaces, we fix a point of an adic space and consider valuated valuation rings inside its residue field.

\begin{definition}\label{defVvandcenter}(Huber {\cite[Definition 3.11.8]{Hu1993}})
Let $x$ be a point of an adic space $X$.
\begin{enumerate}
    \item We  call $(x,A,A^+)$ a \emph{valuated valuation ring on $X$ with support $x$} if $(A,A^+)$ is a valuated valuation ring such that $A\subseteq k(x), A^+\subseteq k(x)^+$ and $\fr(A)=k(x)$.
    \item  A \emph{center} of this valuated valuation ring on $X$ is a specialization $y$ of $x$ such that the canonical map $\cO_{X,y}\rightarrow k(x)$ factors through $A$ and the induced homomorphism $\cO_{X,y}\rightarrow A$ is a local homomorphism of valuated local rings $(\cO_{X,y}, \cO^+_{X,y})$ and $(A,A^+)$.
\end{enumerate}
\end{definition}

The next condition on the existence and uniqueness of centers on affinoids plays a central role in the proofs below.

\begin{lemma}\label{uniquecenteraffinoid}(Huber {\cite[Lemma 3.11.9]{Hu1993}})
Let $X$ be an affinoid adic space and let $(x,A,A^+)$ be a valuated valuation ring on $X$. 
Then $(x,A,A^+)$ has a center on $X$ if and only if, under the canonical map $\cO_X(X)\rightarrow k(x)$, the image of $\cO_X(X)$ is contained in $A$ and the image of $\cO^+_X(X)$ is contained in $A^+$. 
If such a center exists, it is uniquely determined.
\end{lemma}

We collect several basic facts about valuated valuation rings that will be used repeatedly.

\begin{remark}(Huber {\cite{Hu1993}})\label{remarkVvrings}
If $A\subseteq B$ are valuation rings with the same fraction field, there is a unique valuation on $B$ for which $B^+=A$ and $(B,A)$ is a valuated valuation ring. 
For a valuated valuation ring $(A,A^+)$, $A^+$ is a valuation ring as well with $\fr(A^+)=\fr(A)$. 
Let $f\colon X\rightarrow Y$ be a morphism of adic spaces, $x$ a point of $X$ and $(x,A,A^+)$ a valuated valuation ring on $X$. 
Then $(f(x),A\cap k(f(x)), A^+ \cap k(f(x)))$ is a valuated valuation ring on $Y$ which we denote by $f_\ups(A)$. 
If $z$ is a center of $(x,A,A^+)$ on $X$, it follows that $f(z)$ is a center of $f_\ups(A)$ on $Y$.
Conversely, if there exists a valuated valuation ring $(B,B^+)$ in $k(f(x))$, we can choose a valuated valuation ring $(A,A^+)$ on $X$ such that $f_\ups(A)=(B,B^+)$.
\end{remark}

The following proposition provides the structural rationale for our vertical-horizontal approach. 
We write $x \rightsquigarrow y$ if $y$ is a specialization of $x$, i.e., $y \in \overline{\{x\}}$.
 
\begin{proposition}\label{propfactorizationspecial}
Let $x,y$ be points of an affinoid adic space $\Spa(A,A^+)$ such that $x \rightsquigarrow y$. 
Then there exists a point $z$ in $\Spa(A,A^+)$ such that $$
x \rightsquigarrow z \rightsquigarrow y,
$$ where $x\rightsquigarrow z$ is a vertical specialization and $z\rightsquigarrow y $ is a horizontal specialization.
   
\begin{proof}
By {\cite[Corollary 1.1.17]{Hu1993}}, the specialization factors as $x\rightsquigarrow z \rightsquigarrow y$ in $\Spv(A)$.
Since $x\in \Spa(A,A^+)$ and vertical specializations preserve the defining conditions of the adic space, it follows that $z\in \Spa(A,A^+)$.
\end{proof}
\end{proposition}

Specializations can be represented by valuated valuation rings in a precise sense. 
In particular, motivated by the preceding factorization we distinguish two types of valuated valuation rings depending on whether the center is a vertical or a horizontal specialization of its support.

\begin{lemma}(Huber {\cite[Lemma 3.11.10]{Hu1993}})\label{linkspecialVvrings}
Let $x,y$ be points of an adic space $X$ such that $x\rightsquigarrow y$. 
Then there is a valuated valuation ring with support $x$ and center $y$.
More precisely, this valuated valuation ring is of the form $(x,k(x),A^+)$ with $A^+\subseteq k(x)^+$ (resp. of the form $(x,A,k(x)^+)$ with $k(x)^+\subseteq A \subseteq k(x)$) with center $y$ if and only if $y$ is a vertical (resp. horizontal) specialization of $x$.
\end{lemma}

This observation naturally leads to the following definition.

\begin{definition}\label{defhorizontalverticalVV}
Let $x$ be a point of an adic space $X$. 
A \emph{vertical valuated valuation ring} with support $x$ is a valuated valuation ring of the form $(x,k(x),A^+)$ with $A^+\subseteq k(x)^+$.
A \emph{horizontal valuated valuation ring} with support $x$ is a valuated valuation ring of the form $(x,A,k(x)^+)$ with $k(x)^+\subseteq A\subseteq k(x)$. 
\end{definition}

The test objects considered in the analytic criteria \cite[Definition 1.3.5]{Hu1996} are only required to admit vertical specializations as centers, which is why they are chosen to be vertical valuated valuation rings. 
For horizontal valuated valuation rings, a center can be described explicitly as follows.

\begin{definition}\label{defhorizontalcenter}
Let $(x,A,k(x)^+)$ be a horizontal valuated valuation ring on an adic space $X$, and let $\m_A$ be the maximal ideal of $A$. 
A \emph{horizontal center} of $(x,A,k(x)^+)$ on $X$ is a horizontal specialization $y$ of $x$ such that $k(y)$ embeds into $A/\m_A$ and $k(y)^+$ is the intersection of $k(y)$ with the image of $k(x)^+$ in $A/\m_A$. 
\end{definition}

To justify treating the vertical and horizontal criteria separately, we have to demonstrate that they can be recombined to recover the general case. 
The key step is provided by the following statement, which shows that any valuated valuation ring decomposes into a vertical and a horizontal component.

\begin{lemma}\label{generalishorizontalandvertical}
Let $(x,A,A^+)$ be a valuated valuation ring with center $y$ on an affinoid adic space $X$. 
Then $(x,k(x),A^+)$ is a vertical valuated valuation ring with center $z$, and $y$ is a horizontal specialization of $z$. 
Moreover, $(z,A\cap k(z),k(z)^+)$ is a horizontal valuated valuation ring with center $y$.

\begin{proof}
The first assertion is \cite[Lemma 2.4 iv)]{solodov2025relativeverticalcompactificationweakly}, so it remains to show that $(z,A\cap k(z),k(z)^+)$ is a horizontal valuated valuation ring with center $y$. 
Since the specialization $z$ of $x$ is vertical, we obtain the map $k(z)\hookrightarrow k(x)$ and thus $k(z)^+\subseteq A \cap k(z) \subseteq k(z)$. 
Hence, $(z,A\cap k(z), k(z)^+)$ is a valuated valuation ring by Remark \ref{remarkVvrings}, and it is horizontal by Definition \ref{defhorizontalverticalVV}. 
Moreover, it has a center by Definition \ref{uniquecenteraffinoid} as the maps $\cO_X(X)\hookrightarrow A$ and $\cO_X^+(X)\hookrightarrow A^+$ factor through $A\cap k(z)$ and $k(z)^+$, respectively.
Since both the composition and the second morphism of $\cO_{X,y}\rightarrow A \cap k(z) \rightarrow A$ are local, it follows that the first morphism is local as well. 
Lastly, the composition $\cO^+_{X,y}\rightarrow \cO^+_{X,z}\rightarrow k(z)^+$ is local by \cite[Corollary 3.11.7]{Hu1993}, because $y$ is a horizontal specialization of $z$. 
Therefore, according to Definition \ref{defVvandcenter} $y$ is the center of $(z,A \cap k(z),k(z)^+)$. 
\end{proof}
\end{lemma} 

We conclude this section by establishing the terminology for the relevant properties of morphisms.
This slightly refines the classical notions presented in \cite[Section 1.3]{Hu1996} by decoupling them into vertical and horizontal counterparts.

\begin{definition}
Let $f\colon X\rightarrow Y$ be a morphism of adic spaces. We call $f$ \begin{enumerate}[label=(\roman*)]
    \item \emph{vertically}  (resp. \emph{horizontally}) \emph{specializing at a point $x\in X$}, if every vertical (resp. horizontal) specialization $y\p$ of $f(x)$ lifts to a vertical (resp. horizontal) specialization $x\p$ of $x$ with $f(x\p)=y\p$. 
    \item \emph{universally vertically} (resp. \emph{horizontally}) \emph{specializing} at $x\in X$, if $f$ is locally of weakly finite type and for every adic morphism $Y\p\rightarrow Y$ of adic spaces and every point $x\p$ of $X\times_Y Y\p$ that lies over $x$, the projection $X\times_Y Y\p\rightarrow Y\p$ is vertically (resp. horizontally) specializing at $x\p$. 
    \item \emph{(universally) vertically} (resp. \emph{horizontally}) \emph{specializing}, if it has this property at every point of $X$. 
    \item \emph{vertically} (resp. \emph{horizontally}) \emph{separated}, if $f$ is quasi-separated, locally of weakly finite type and for every vertical (resp. horizontal) valuated valuation ring and every point $y$ of $Y$ there is at most one center $z$ on $X$ such that $f(z)=y$. 
    \item \emph{vertically} (resp. \emph{horizontally}) \emph{partially proper}, if $f$ is locally of $^+$weakly finite type, vertically (resp. horizontally) separated and universally vertically (resp. horizontally) specializing. 
    \item \emph{vertically} (resp. \emph{horizontally}) \emph{proper}, if $f$ is quasi-compact and vertically (resp. horizontally) partially proper. 
    \end{enumerate}
\end{definition}

\section{Vertical valuative criteria}\label{sectionverticalVc}
In the analytic setting, all specializations are vertical, and Huber’s valuative criteria already provide a complete description. 
We slightly generalize his results to arbitrary adic spaces by focusing strictly on vertical properties.
These statements are termed \emph{vertical} valuative criteria to emphasize their role as one half of the general picture. 

\begin{proposition}[Vertical valuative criteria]\label{verticalVc}
Let $f\colon X\rightarrow Y$ be a morphism of adic spaces which is locally of weakly finite type and quasi-separated (resp. locally of weakly finite type, resp. (locally) of $^+$weakly finite type and quasi-separated). 
The following are equivalent. \begin{enumerate}
        \item $f$ is vertically separated (resp. universally vertically specializing, resp. vertically (partially) proper),
        \item for every affinoid field $(k,k^+)$ and every valuation ring $A\subseteq k^+$ in a commutative diagram \begin{equation*}
            \begin{tikzcd}
         {\Spa(k,k^+)} \arrow[d] \arrow[r]     & X \arrow[d] \\
         {\Spa(k,A)} \arrow[r] \arrow[ru, dotted, "p"] & Y,          
           \end{tikzcd}\end{equation*}
           there is at most one (resp. there exists a, resp. there exists exactly one) morphism $p$ making the diagram commute. 
    \end{enumerate}
    
\begin{proof}
The proofs of \cite[Proposition 1.3.7 and 1.3.8]{Hu1996} carry over verbatim to the general setting, if one inserts \emph{vertical} everywhere. 
\end{proof}
\end{proposition}

\section{Horizontal valuative criteria}\label{sectionhorizontalVc}
This section contains the remaining component necessary to establish the main result of this paper: valuative criteria for horizontal separatedness and horizontal (partial) properness. 
Since horizontal specializations are trivial in analytic spaces, these criteria specifically target non-analytic phenomena where Huber's original criteria are no longer exhaustive. 
We begin with two essential preparatory observations.

\begin{remark}\label{remarklocalhuberpair}
Let $(k,k^+)$ be an affinoid field and $B$ a subring satisfying $k^+\subseteq B \subseteq k$ with maximal ideal $\m_B$.
As recalled from \cite[Lemma 6.8]{Hubner2021AdicTameSite}, the adic space $\Spa{(B,k^+)}$ possesses a unique closed point, which corresponds to the valuation ring $k^+/\m_B$, and every point specializes to it.
This property will be used repeatedly when interpreting centers of horizontal valuated valuation rings. 
\end{remark}

\begin{remark}\label{remarkwlogchangecomdiag}
Let $X\rightarrow Y$ be a morphism of adic spaces and $x$ a point of $X$. 
Consider a commutative diagram \begin{center}\begin{tikzcd}
        S \arrow[r]\arrow[d] &X \arrow[d] \\
        T \arrow[r] &Y
    \end{tikzcd}
    \end{center} 
in which $S=\Spa{(k(x),k(x)^+)}$ and $T=\Spa{(B,k(x)^+)}$ with $k(x)^+\subseteq B \subseteq k(x)$. 
Note that in this context the uniqueness (resp. existence) of a lift is equivalent to the same property for diagrams in which $S=\Spa{(k,k^+)}$ and $T=\Spa{(B',k^+)}$ with $(k,k^+)$ an affinoid field and $k^+\subseteq B'\subseteq k$, where $x$ is the image of the closed point under $S\rightarrow X$.
Hence, horizontal valuative lifting problems may be checked in residue fields without loss of generality. 
\end{remark}

The following theorem provides a horizontal analogue to Proposition \ref{verticalVc}. It characterizes horizontal properties of morphisms of adic spaces by corresponding lifting properties for horizontal valuated valuation rings.

\begin{theorem}[Horizontal valuative criteria]\label{horizontalVc}
Let $f\colon X\rightarrow Y$ be a morphism of adic spaces which is locally of weakly finite type and quasi-separated (resp. locally of weakly finite type, resp. (locally) of $^+$weakly finite type and quasi-separated). 
The following are equivalent. \begin{enumerate}
        \item $f$ is horizontally separated (resp. universally horizontally specializing, resp. horizontally (partially) proper),
        \item for every horizontal valuated valuation ring with support $x$ on $X$ and for every center $y$ of its pushforward to $Y$, there is at most one (resp. there exists a, resp. there exists a unique) center $z$ on $X$ such that $f(z)=y$, 
        \item for every affinoid field $(k,k^+)$ and every ring $k^+\subseteq B\subseteq k$ in a commutative diagram \begin{equation*}
            \begin{tikzcd}
         {\Spa(k,k^+)} \arrow[d] \arrow[r]     & X \arrow[d, "f"] \\
         {\Spa(B, k^+)} \arrow[r] \arrow[ru, dotted, "p"] & Y,          
           \end{tikzcd}\end{equation*}
           there is at most one (resp. there exists a, resp. there exists exactly one) morphism $p$ making the diagram commute. 
    \end{enumerate}

\begin{proof}
\textit{(1) $\iff$ (2):} 
By definition, the assertion for horizontal (partial) properness follows once the cases of horizontal separatedness and the universal horizontal specialization property are established. 
   
The case of horizontal separatedness is immediate from the definition. 
    
For the other claim, suppose first that $f$ is universally horizontally specializing. We show that every horizontal valuated valuation ring with a center on $Y$ has a center on $X$ above it.
Let $(x,B,k(x)^+)$ be a horizontal valuated valuation ring on $X$ whose image has center $y$ on $Y$. 
This setup induces a commutative diagram
    \begin{equation*}
      \begin{tikzcd}
        {\Spa(k(x),k(x)^+)} \arrow[d, "f\p"] \arrow[r, "g\p"]     & X \arrow[d, "f"] \\
        {\Spa(B, k(x)^+)} \arrow[r, "g"] & Y . 
    \end{tikzcd}
    \end{equation*}
Let $x\pp$ be the closed point of $\Spa(k(x),k(x)^+)$.  
Since its residue field $k(x\pp)$ agrees with the completed residue field $\com{k(x)}$ of $x$, we obtain the valuated valuation ring $(x\pp,\com{B},\com{k(x)^+})$ on $\Spa(k(x),k(x)^+)$.
Let $\ti{x}=f\p(x\pp)$. 
Note that the residue field of $\ti{x}$ is again $\com{k(x)}$. 
Now, the horizontal valuated valuation ring $(\ti{x}, \com{B}, \com{k(x)^+})$ on $\Spa(B,k(x)^+)$ has a unique center, which we denote by $\ti{y}$.
Since $g\p(x\pp)=x$, pushing forward the valuated valuation ring along $g\p$ gives $(x, B, k(x)^+)$ on $X$. 
Due to commutativity of the diagram, its image on $Y$ agrees with the pushforward of $(\ti{x}, \com{B}, \com{k(x)^+})$ along $g$. 
By construction, we then have $g(\ti{y})=y$. 
Note that $g$ is an adic morphism.
Since $f$ is locally of weakly finite type, we therefore may consider the fiber product $X\times_Y \Spa(B,k(x)^+)$. 
By its universal property, the preceding diagram induces a unique morphism $\gamma$ to $X\times_Y \Spa(B,k(x)^+)$, together with the two projections $\pi_1, \pi_2$, making the following diagram cartesian     
   \begin{center}\begin{tikzcd}
    \Spa(k(x),k(x)^+) \arrow[rd, dashed, "\gamma"] \arrow[rdd, bend right, "f\p"] \arrow[rrd, bend left, "g\p"] & & \\& X\times_Y\Spa(B,k(x)^+)\arrow[d, "\pi_1"] \arrow[r, "\pi_2"] & X \arrow[d, "f"] \\ & \Spa(B,k(x)^+) \arrow[r, "g"] &Y.   
   \end{tikzcd} \end{center}
Setting $x\p=\gamma(x\pp)$, we have $\pi_1(x\p)=\ti{x}$ and $\pi_2(x\p)=x$. 
Moreover, since $k(x\p)=\com{k(x)}$, we obtain the horizontal valuated valuation ring $\gamma_\ups(\com{B})=(x\p, \com{B}, \com{k(x)^+})$ on $X\times_Y \Spa(B,k(x)^+)$. 
Because $\ti{y}$ is the center of a horizontal valuated valuation ring with support $\ti{x}$, it is a horizontal specialization of $\ti{x}$. 
As $f$ is universally horizontally specializing, the base change $\pi_1$ is horizontally specializing. 
Hence, there exists a horizontal specialization $y\p$ of $x\p$ such that $\pi_1(y\p)=\ti{y}$. 
By construction, $y\p$ also is a center of the induced horizontal valuated valuation ring $(x\p, \com{B}, \com{k(x)^+})$ on $X\times_Y \Spa(B,k(x)^+)$. 
Then $\pi_2(y\p)\eqqcolon z$ is a center of $(x,B,k(x)^+)$ on $X$ satisfying $f(z)=y$, as desired.
    
Conversely, assume the existence of centers as in (2). We want to prove that $f$ is universally horizontally specializing. Consider a cartesian diagram 
    \begin{equation*}
           \begin{tikzcd}
      X\p \arrow[d, "f\p"] \arrow[r, "g\p"] & X \arrow[d, "f"] \\
     \ti{Y} \arrow[r, "g"]    & Y.         
      \end{tikzcd}
    \end{equation*}
Let $x\p$ be an arbitrary point of $X\p$ and set $x\coloneqq g\p(x\p)$. 
Consider a horizontal specialization $\ti{y}$ of $\ti{x}\coloneqq f\p(x\p)$ in $\ti{Y}$. 
With Lemma \ref{linkspecialVvrings}, applied in the horizontal case, we choose a horizontal valuated valuation ring $(\ti{x},\ti{B}, k(\ti{x})^+)$ with center $\ti{y}$ on $\ti{Y}$. 
Moreover, using Remark \ref{remarkVvrings} we choose a lift to a horizontal valuated valuation ring $(x\p,B\p,k(x\p)^+)$ on $X\p$ such that $f_\ups\p(B\p)=(\ti{x},\ti{B}, k(\ti{x})^+)$. 
Let $g_\ups\p(B\p)\eqqcolon(x,B,k(x)^+)$ and $g_\ups(\ti{B})\eqqcolon(x\pp,B\pp, k(x\pp)^+)$ be the horizontal valuated valuation rings on $X$ and $Y$, respectively.
Commutativity then yields $f_\ups(B)=(x\pp,B\pp, k(x\pp)^+)$. 
Hence, the point $g(\ti{y})\eqqcolon y$ is a horizontal center of the induced valuated valuation ring $f_\ups(B)$ on $Y$. 
By assumption, we obtain a horizontal center $z$ of $(x,B,k(x)^+)$ on $X$ such that $f(z)=y$. 
Now, we choose affinoid neighbourhoods $U\subseteq X, \ti{V}\subseteq \ti{Y}$ and $V\subseteq Y$, with $z\in U$ and $\ti{y}\in \ti{V}$, such that $f(U)\subseteq V$ and $g(\ti{V})\subseteq V$. 
Consider the affinoid fiber product $U\times_V \ti{V}$. 
Since $x\p\in f^{\prime{-1}}(\ti{V})\cap g^{\prime{-1}}(U)$, the valuated valuation ring $(x\p,B\p,k(x\p)^+)$ may be regarded as a valuated valuation ring on $U\times_V \ti{V}$. 
Note that $(x\p,B\p,k(x\p)^+)$ has a unique center $y\p$ on $U\times_V \ti{V}$ as well, which is a horizontal specialization of $x\p$. 
Consequently, $f\p(y\p)$ is a center of $(\ti{x},\ti{B}, k(\ti{x})^+)$ on $\ti{V}$. 
By uniqueness of centers on affinoid spaces, we have $f\p(y\p)=\ti{y}$. 
This shows universal horizontal specialization of $f$.
    
\textit{(2) $\implies$ (3):} 
We first prove the uniqueness statement. 
Assume that every horizontal valuated valuation ring has at most one center above a prescribed center on $Y$. 
Consider a horizontal test diagram as in (3). 
Let $x$ be the image of the closed point under $\Spa(k,k^+)\rightarrow X$. 
By Remark \ref{remarkwlogchangecomdiag}, we may exchange $\Spa(k,k^+)$ with $\Spa(k(x),k(x)^+)$ and $\Spa(B,k^+)$ with $\Spa(B\p,k(x)^+)$ where $k(x)^+\subseteq B\p \subseteq k(x)^+$.
Let $\ti{x}$ be the image of the closed point under $\Spa(k(x),k(x)^+)\rightarrow \Spa(B\p,k(x)^+)$ and let $\ti{y}$ be the unique closed point of $\Spa(B\p,k(x)^+)$ as described in Remark \ref{remarklocalhuberpair}.
Now, we consider the horizontal valuated valuation ring $(\ti{x},\com{B\p},\com{k(x)^+})$ with center $\ti{y}$ on $\Spa(B\p,k(x)^+)$.
Its pushforward on $Y$ has the image of $\ti{y}$ as a center, which we denote by $y$. 
Suppose now there are two distinct morphisms $p,q:\Spa(B\p,k(x)^+)\rightarrow X$ making the diagram commute. 
Note that we necessarily have $p(\ti{y})\neq q(\ti{y})$.
The points $p(\ti{y})$ and $q(\ti{y})$ therefore constitute two different centers of the induced horizontal valuated valuation ring $(x,B\p,k(x)^+)$ on $X$ satisfying $f(p(\ti{y}))=y=f(q(\ti{y}))$. 
This contradicts the uniqueness property of centers assumed in (2).
        
If (2) guarantees the existence of a center with the described properties, it is straightforward to construct the wanted lift using Lemma \ref{uniquecenteraffinoid}.
     
\textit{(3) $\implies$ (2):} 
Assume first that any horizontal test diagram as above admits at most one diagonal morphism.
Let $(x,B,k(x)^+)$ be a horizontal valuated valuation ring on $X$ with center $y$ on $Y$. 
This yields a horizontal test diagram as described in (3).
Suppose there are two distinct centers $z_1$ and $z_2$ of $(x,B,k(x)^+)$ on $X$ such that $f(z_1)=y=f(z_2)$. 
Choose affinoid neighbourhoods $U$ of $z_1$ and $V$ of $z_2$.
The center condition by Lemma \ref{uniquecenteraffinoid} then gives morphisms $\Spa(B,k(x)^+)\longrightarrow U$ and $\Spa(B,k(x)^+)\longrightarrow V$.
Composing with the respective inclusions into $X$ yields two diagonal morphisms   $\Spa(B,k(x)^+)\longrightarrow X$ in the horizontal test diagram. 
These two morphisms are distinct, since the unique closed point of $\Spa(B,k(x)^+)$ is mapped to $z_1$ by the first morphism and to $z_2$ by the second. 
This contradicts the uniqueness assumed in (3). 
Hence, there is at most one center of $(x,B,k(x)^+)$ on $X$ above the prescribed center $y$ on $Y$. 
      
For existence, assume that any horizontal test diagram as described has a lifting morphism.
Let $(x,B,k(x)^+)$ again be a horizontal valuated valuation ring on $X$ whose image has a center $y$ on $Y$.
Consider the horizontal test diagram determined by the support $x$ and by the center $y$ of the pushforward on $Y$.
By assumption, it admits a diagonal lift $p$.
Let $\tilde y$ be the unique closed point of $\Spa(B,k(x)^+)$. 
Then $z\coloneqq p(\tilde y)$ is a center of $(x,B,k(x)^+)$ on $X$. 
Due to commutativity of the diagram we have $f(z)=y$, so $z$ indeed serves as the required center.
    \end{proof}
    \end{theorem}
    
\section{Unified valuative criteria}\label{sectionunifiedVc}
Combining the vertical criteria of Section~\ref{sectionverticalVc} with the horizontal criteria of Section~\ref{sectionhorizontalVc}, we finally obtain unified valuative criteria valid for all adic spaces. 
For the usual notions of separatedness, the universal specialization property and (partial) properness, we refer to \cite[Definitions 1.3.1-1.3.3 and Lemma 1.3.4]{Hu1996}.

\begin{theorem}[Unified valuative criteria]\label{unifiedVc}
Let $f:X\rightarrow Y$ be a morphism of adic spaces which is locally of weakly finite type and quasi-separated (resp. locally of weakly finite type, resp. (locally) of $^+$weakly finite type and quasi-separated). 
The following are equivalent. \begin{enumerate}
    \item $f$ is separated (resp. universally specializing, resp. (partially) proper),
    \item for every affinoid field $(k,k^+)$ and all commutative diagrams 
    \begin{equation*}
       \begin{gathered}
        \begin{tikzcd}[ampersand replacement=\&]
                \Spa(k, k^+) \arrow[r] \arrow[d] \& X \arrow[d] \\
                \Spa(k, A) \arrow[r] \arrow[ur, dotted, "p"] \& Y,
        \end{tikzcd} \\
        \text{with } A \subseteq k^+,
        \end{gathered}
        \qquad \text{and} \qquad
        \begin{gathered}
        \begin{tikzcd}[ampersand replacement=\&]
                \Spa(k, k^+) \arrow[r] \arrow[d] \& X \arrow[d] \\
                \Spa(B, k^+) \arrow[r] \arrow[ur, dotted, "q"] \& Y,
        \end{tikzcd} \\
        \text{with } k^+ \subseteq B \subseteq k,
        \end{gathered}
    \end{equation*}
    there is at most one (resp. there exists a, resp. there exists exactly one) morphism $p$ resp. $q$ making the diagram commute.
\end{enumerate}

\begin{proof}
\textit{(1) $\implies$ (2):} 
Since (partial) properness is defined as separatedness together with universal specialization under the relevant finiteness hypotheses, it suffices to treat these two properties separately.
  
Suppose first that $f$ is separated. 
By \cite[Proposition 3.11.12]{Hu1993}, every valuated valuation ring on $X$ has at most one center above any given center on $Y$.
As vertical and horizontal valuated valuation rings are special cases of valuated valuation rings, this property is inherited. 
The vertical and horizontal lifting uniqueness conditions therefore follow from Proposition \ref{verticalVc} and Theorem \ref{horizontalVc}. 
  
Second, let $f$ be universally specializing. It is shown in \cite[Proposition 3.4]{solodov2025relativeverticalcompactificationweakly} that $f$ is universally vertically specializing. 
The same argument applied to horizontal specializations shows that
$f$ is also universally horizontally specializing. 
Proposition \ref{verticalVc} and Theorem \ref{horizontalVc} then yield the required lifting conditions. 
  
\textit{(2) $\implies$ (1):} 
Conversely, assume the uniqueness parts of the vertical and horizontal lifting conditions. 
We want to show that $f$ is separated.
Let $(x,A,A^+)$ be a valuated valuation ring on $X$, and let $y$ be a center of its pushforward to $Y$. 
Suppose there are two centers $z_1,z_2$ of $(x,A,A^+)$ on $X$ with $f(z_1)=y=f(z_2)$. 
Choose an open affinoid neighbourhood $V\subseteq Y$ of $y$.
Moreover, choose open affinoid neighbourhoods $U_1,U_2\subseteq X$ of $z_1,z_2$, respectively, such that $f(U_1)\subseteq V$ and $f(U_2)\subseteq V$. 
Since $z_1,z_2$ are specializations of $x$, the point $x$ and thus the valuated valuation ring $(x,A,A^+)$ lies in $U_1$ and $U_2$ with its respective center.
By Lemma \ref{generalishorizontalandvertical}, the associated vertical valuated valuation ring $(x,k(x),A^+)$ has centers $v_1$ on $U_1$ and $v_2$ on $U_2$ such that $z_i$ is a horizontal specialization of $v_i$ for each $i=1,2$. 
By Remark \ref{remarkVvrings}, the points $f(v_1)$ and $f(v_2)$ are centers of the pushforward of $(x,k(x),A^+)$ on $V$.
Since centers on affinoids are uniquely determined, we obtain $f(v_1)=f(v_2)$.
The vertical uniqueness condition then implies $v_1=v_2$, and we denote this common point by $v$. 
Next, by Lemma \ref{generalishorizontalandvertical}, the corresponding horizontal valuated valuation ring $(v,A \cap k(v), k(v)^+)$ has centers $z_1$ on $U_1$ and $z_2$ on $U_2$, respectively. 
Similarly, $f(z_1)=y=f(z_2)$ is a center of the image of $(v,A \cap k(v), k(v)^+)$ on $V$. 
Applying the horizontal uniqueness condition now yields $z_1=z_2$. 
This shows that $f$ is separated by \cite[Proposition 3.11.12]{Hu1993}.
  
Finally, assume the existence parts of both the vertical and horizontal  lifting conditions. 
By Proposition \ref{verticalVc} and Theorem \ref{horizontalVc}, this implies that $f$ is both universally vertically and universally horizontally specializing.
We now wish to prove that $f$ is universally specializing.
Let $f\p\colon X \times_Y \ti{Y}\rightarrow \ti{Y}$ be a base change of $f$.
We choose an arbitrary point $x\p$ of $X \times_Y\ti{Y}$.
Consider a specialization $\ti{y}$ of $\ti{x}\coloneqq f\p(x\p)$ in $\ti{Y}$.
By Proposition \ref{propfactorizationspecial}, it factors through a vertical specialization $\ti{z}$ of $\ti{x}$ such that $\ti{y}$ is a horizontal specialization of $\ti{z}$.
By assumption, there is a vertical specialization $z\p$ of $x\p$ with $f\p(z\p)=\ti{z}$.
In the same way, the analogous statement provides a horizontal specialization $y\p$ of $z\p$ with $f\p(y\p)=\ti{y}$. 
Thus, we found a specialization $y\p$ of $x\p$ in $X \times_Y\ti{Y}$ such that $f\p(y\p)=\ti{y}$, which verifies that $f$ is universally specializing.
 \end{proof}
 \end{theorem}

\bibliographystyle{alpha}
\bibliography{valuativecriteria/lit}
\end{document}